\documentclass[11pt]{amsart}
\usepackage{geometry}                
\usepackage{graphicx}
\usepackage{amssymb}
\usepackage{epstopdf}

\usepackage{amsmath,amssymb}
\usepackage{enumerate}
\usepackage{skak}
\usepackage{tikz}
\usetikzlibrary{arrows.meta,automata,quotes}
\usetikzlibrary{decorations.pathreplacing,calc}
\usetikzlibrary{backgrounds}
\usepackage{pgfplots}

\usepackage[colorlinks]{hyperref}
\usepackage[T1]{fontenc}
\usepackage[utf8]{inputenc}
\usepackage{mathrsfs}  
\usepackage{caption}
\theoremstyle{plain}
\newtheorem{thm}{Theorem}
\newtheorem{corollary}[thm]{Corollary}

\newtheorem{lemma}[thm]{Lemma}

\newtheorem{thmm}{Theorem}
 
\theoremstyle{definition}

\newtheorem{question}[thm]{Question}
\newtheorem{remark}[thm]{Remark}
\newtheorem{example}[thm]{Example}

\newcommand{\N}{\ensuremath{\mathbb{N}}}
\renewcommand{\P}{\ensuremath{\mathscr{P}}}
\newcommand{\W}{\mathcal W}

\renewcommand{\ge}{\geqslant}
\renewcommand{\le}{\leqslant}
\renewcommand{\geq}{\geqslant}
\renewcommand{\leq}{\leqslant}

\title{Additive subtraction games}
\author{Urban Larsson}
\address{Department of Industrial Engineering and Operations Research, IIT Bombay, India}
\email{larsson@iitb.ac.in}
\author{Hikaru Manabe}
\address{College of Information Science, School of Informatics, University of Tsukuba, Japan}
\email{urakihebanam@gmail.com}
\begin{document}

\begin{abstract}
We determine the full nim-value structure of additive subtraction games in the {\em primitive quadratic} regime. The problem appears in Winning Ways by Berlekamp et al. in 1982; it includes a closed formula, involving Beatty-type {\em bracket expressions} on rational moduli, for determining the \P-positions, but to the best of our knowledge, a complete proof of this claim has not yet appeared in the literature; Miklós and Post (2024) established outcome-periodicity, but without reference to that closed formula. The {\em primitive quadratic} case captures the source of the quadratic complexity of the problem, a claim supported by recent research in the dual setting of sink subtraction by Bhagat et al. This study focuses on a number theoretic solution involving the classical closed formula, and we establish that each nim-value sequence resides on a linear shift of the classical \P-positions. 
\end{abstract}
\date{\today}
\maketitle

\section{Introduction}

Imagine a heap of counters sitting between two players, who take turns 
removing exactly $a$, $b$, or $a+b$ counters, where $0<a<b$ are given 
integers. The player who cannot move, because every legal removal would 
make the heap negative, loses. Simple to state, yet the question of who 
wins from any given heap size turns out to conceal surprisingly deep 
arithmetic.

This is an instance of a {\sc subtraction game} \cite{berlekamp2004winning}, 
in the sub-class of {\sc additive subtraction}. The tool for answering the 
winning question, not just for a single heap, but for any number of heaps 
played together, is the \emph{nim-value} (or nimber) of a position. When 
two or more such games are played in parallel, with a player choosing on each 
turn which component game to move in, the combined game is called a 
\emph{disjunctive sum}. See Figure~\ref{fig:play}.

\begin{figure}[ht]
    \centering
    \includegraphics[width=0.8\linewidth]{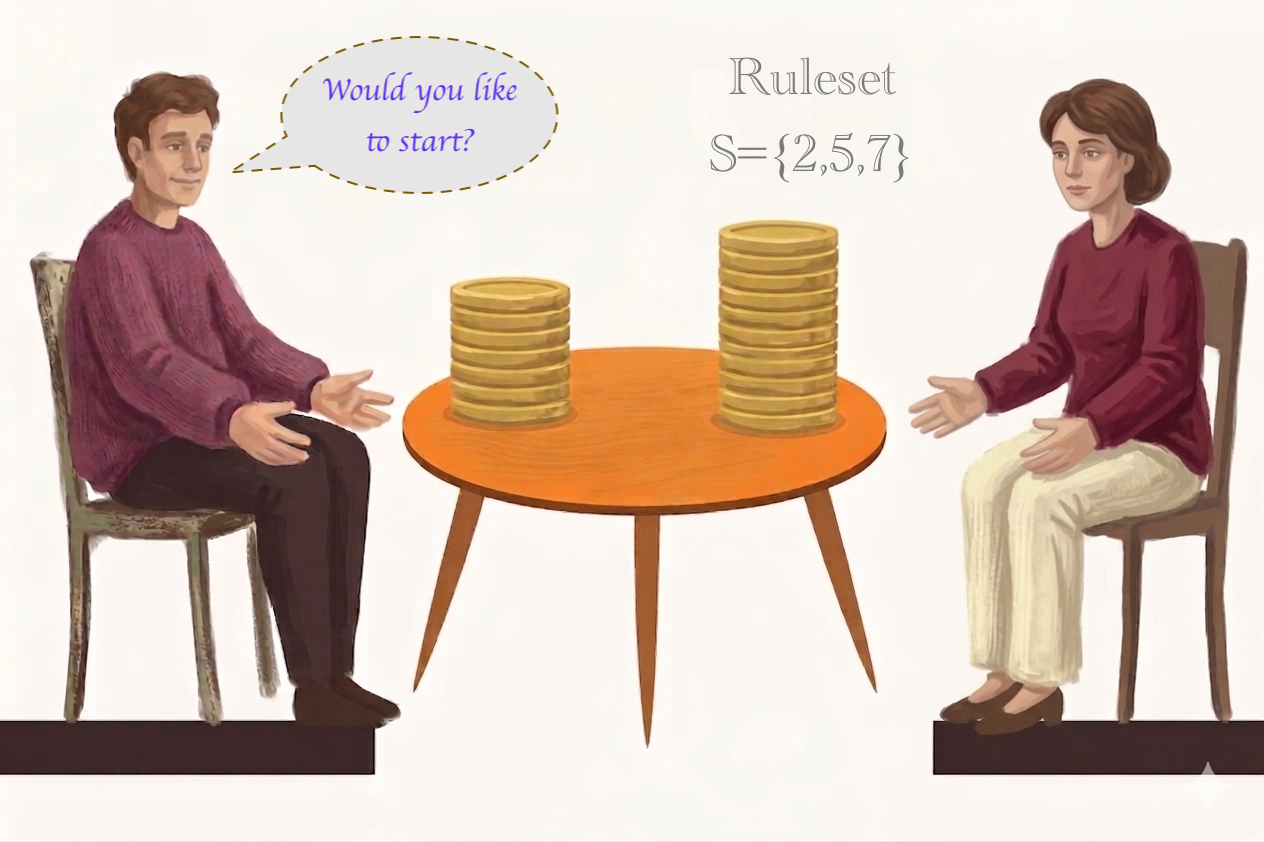}
    \caption{A position of {\sc additive subtraction}, with ruleset $S=\{2,5,7\}$,
    consisting of two heaps of sizes seven and eleven, played in a disjunctive sum. The player making the last move wins.
    Should Alice accept Bob's offer to move first?}
    \label{fig:play}
\end{figure}

The famous Sprague--Grundy theory \cite{bouton1901nim, Sprague1935, grundy1939mathematics} tells us that every such component has a nim-value, a single non-negative integer, and the nim-value of the disjunctive sum is simply the bitwise XOR of the nim-values of its components. Suppose, for example, that you face two heaps, perhaps of different sizes, but both with nim-value three; the theory reveals that the position is losing for the player to move, independently of which heap they touch. The nim-values therefore reveal a complete winning strategy across any combination of rulesets and heap sizes.

The nim-value of a position is assigned recursively: it is the smallest 
non-negative integer not among the nim-values of all the move options; this is the classical \emph{mex} (minimum excludant) rule. We think of it in terms of two components: for each position, one must verify 
both \emph{anti-collision} (the same value does not appear among the options) 
and \emph{reachability} (every smaller value appears among the options). For 
example, if a position can reach nim-values $\{0,1,3\}$, its own nim-value is 
two. Positions of nim-value zero are the \emph{\P-positions}, the losing 
positions for the player to move; they are justified by anti-collision alone. 
For our three-move game $S=\{a,b,a+b\}$, the mex-rule implies all nim-values 
lie in $\{0,1,2,3\}$, since at most three options are available from any 
position. The question is: ``which positions get which value?''

The recursive algorithm answers this for small heap sizes. Let us illustrate 
it for the game in Figure~\ref{fig:play}, computing just enough to guide 
Alice towards the right answer to Bob's question.

\begin{table}[h]
\caption{Nim-values for the ruleset~$S=\{2,5,7\}$.}
\centering
\begin{tabular}{|c|c|c|c|c|c|c|c|c|c|c|c|c|}
\hline
Heap size & 0 & 1 & 2 & 3 & 4 & 5 & 6 & 7 & 8 & 9 & 10 & 11 \\ \hline
Nim-value & 0 & 0 & 1 & 1 & 0 & 2 & 1 & 3 & 2 & 2 & 0  & 3  \\ \hline
\end{tabular}

\end{table}

Both heaps have nim-value three, so their XOR is $3\oplus 3=0$: this is a 
\P-position, and Alice should \emph{decline} Bob's offer. But for large 
heaps, or many heaps, direct computation becomes infeasible. 

We prove here that each nimber sequence has an elegant closed form formula, which makes computation tractable.\footnote{Since it is known that all subtraction games are periodic \cite{golomb1966mathematical}, one could argue that every subtraction game is tractable, even without knowledge of a closed formula. However, Flammenkamp \cite{Flammenkamp_1997} conjectures that there exist sequences of games with exponential period length in terms of $\max S$, which, if true, would slow down any such semi-tractable approach considerably.} The  sequence of \P-positions has the following form in the general case including $S=\{2,5,7\}$. For a given instance of additive subtraction with game parameters $0<a<b$, for all $n\in\N=\{0,1,\ldots\}$, let 
\begin{align}\label{eq:wn}
  w_n := n + a\Bigl\lfloor\frac{n}{a}\Bigr\rfloor +
  b\Bigl\lfloor\frac{n}{\delta}\Bigr\rfloor,
\end{align}
where $\delta = b-a$. A slight generalization of this formula, involving two interlocking floor 
functions, first appeared in \cite{berlekamp2004winning} in 1982, but 
without a proof. To the best of our knowledge, no proof has appeared in 
the literature since, although it appeared in a similar form in \cite{Flammenkamp_1997} in the late 1990s. The present paper closes that gap and goes further: we determine the complete nim-value structure. Note that the nim-value of any position smaller than $a$ is zero, which also follows directly by \eqref{eq:wn}, since both floor terms evaluate to zero. 

We work in the \emph{primitive quadratic regime}, $a<\delta<2a$ with $\gcd(a,\delta)=1$, so we study the ruleset family {\sc primitive quadratic additive subtraction}. This is the first parameter range in which the two scales $a$ and $\delta$ interact non-trivially, and it captures the core complexity of the problem; see also the dual setting of Bhagat et al.~\cite{Bhagat}. We use the following notation throughout. If $X\subset\mathbb{N}$ and $y\in\mathbb{N}$, then
\[
  X+y := \{x+y : x\in X\}
  \quad\text{and}\quad
  X-y := \{x-y : x\in X,\, x\ge y\},
\]
so that $X-y$ contains only non-negative integers. Writing 
$\W_0=\{w_n:n\ge 0\}$, the four nim-value classes are
\[
  \W_0,\quad \W_1=\W_0+a,\quad \W_2=\W_0-b,\quad 
  \W_3=(\W_0-\delta)\setminus\W_0.
\]
In words: three of the four classes are simple shifts of the \P-positions, while the fourth, the nim-value three positions, and the most intricate, is what remains after removing the $\delta$-shift's overlap with $\W_0$.

\begin{thm}[Main Theorem]\label{thm:main}
Assume $a<\delta<2a$ and $\gcd(a,\delta)=1$, with $b=a+\delta$. Then 
$\W_0,\W_1,\W_2,\W_3$ partition $\mathbb{N}$, and $\W_i$ is precisely 
the set of positions of nim-value $i$, for $i\in\{0,1,2,3\}$.
\end{thm}

Three of the four classes are handled with relatively fewer complications: $\W_0$ by a number-theoretic anti-collision argument, $\W_1$ by Ferguson's pairing principle \cite{berlekamp2004winning} (the smallest move $a$ pairs every \P-position with a unique nim-value one position), and $\W_2$ by a straightforward induction on reachability combined with nim-value exclusion. The nim-value three positions in $\W_3$ are the most delicate: we must establish their count per period, and we use a collision-counting argument that forms the arithmetic heart of the paper. 

We call \eqref{eq:wn} a ``bracket expression on three terms'', and more generally, an expression $$\sum_{i=1}^{t}c_i\lfloor n\alpha_i\rfloor,$$ with integer coefficients $c_i$ and positive rational moduli $\alpha_i$, is a {\em bracket expression} on $t$ terms. 

As we will see, the arithmetic structure of the problem ultimately stems from the index-regular gap structure of the \P-position sequence \eqref{eq:wn}: only four gap sizes appear, and their locations are 
determined exactly by the residue classes of the index modulo $a$ and 
$\delta$. The index-period length of \eqref{eq:wn} is $a\delta$, which implies that the 
heap-period length is quadratic in $a$, namely $h:=(3\delta+a)a$. The densities of the four classes 
are therefore
\[
  \frac{\delta}{3\delta+a},\quad\frac{\delta}{3\delta+a},\quad
  \frac{\delta}{3\delta+a},\quad\frac{a}{3\delta+a},
\]
for $\W_0,\W_1,\W_2,\W_3$ respectively. Notably, the nim-value three 
positions are not only the most intricate, they are also the rarest, 
with density strictly less than the other three classes, reflecting the primitive quadratic inequality $a<\delta$.

Outside the primitive quadratic regime, essentially nothing new occurs: either one denominator becomes a multiple of the other in \eqref{eq:wn}, which reduces the problem to linear complexity, or certain pairs of nim-value ``factors''   
expand regularly by $2a$ with each increase of $b$ by $2\delta_0$ (if $\delta_0$ is defined by the primitive quadratic regime), preserving quadratic complexity; see~\cite{Bhagat}.

\begin{remark}
A recent result of Miklós and Post \cite[Theorem~21]{MiklosPost} establishes 
periodicity of the outcome sequence for additive subtraction games 
$\{a,b,a+b\}$; see also \cite[Theorem~4]{YutoMoriwaki} where they study 
nim-value periodicity. These studies concern block constructions, similar 
to \cite{Bhagat}, and do not mention bracket expressions and their affine 
translations for partitioning the natural numbers to describe the nim-values. 
In contrast, the present paper determines the complete nim-value structure 
in the primitive quadratic regime and provides an explicit description of 
all four nim-value classes.
\end{remark}

The paper is organized as follows. Section~\ref{sec:anticoll} establishes 
anti-collision for the candidate set of \P-positions, $\W_0$, proving that no move 
connects any two elements of $\W_0$, and deduces the nim-value one structure of 
$\W_1$, via Ferguson's pairing principle. Section~\ref{sec:two} handles the 
nim-value two and three positions, with the key reachability arguments 
supported by Figure~\ref{fig:winning_model}, which gives a geometric interpretation  
of the central arithmetic mechanism. Section~\ref{sec:numtheo} isolates the 
number-theoretic heart of the paper via a collision-counting argument. We demonstrate  
that, in one heap-period, exactly $ad$ elements of $\W_0$ lie within distance 
$\delta$ of another element of $\W_0$, from which the nim-value three count 
follows. Section~\ref{sec:numexa} illuminates the number-theoretic approach with three 
explicit numerical examples. Section~\ref{sec:gametheo} assembles all ingredients into the proof of Theorem~\ref{thm:main}, and Section~\ref{sec:future} closes with open problems and directions for future 
work. 
\section{Anti-collision of candidate \P-positions}\label{sec:anticoll}
In this section, we prove the first part of the mex-verification, anti-collision, for the candidate set of \P-positions. Note that the  reachability part is empty, since we here study zero values. However, since any `independent subset of the natural numbers' satisfies the anti-collision requirement, the true positioning of the zero values will only be confirmed when we have established the complete nim-value tetra-partitioning of all game positions, as in Theorem~\ref{thm:main}. 

\begin{lemma}[Anti-collision]\label{lem:anticollision}
Let $S=\{a,b,a+b\}$ with $b=a+\delta$ and $a<\delta<2a$.
For all $n\in\N$, let 
\[
  w_n
  =
  n
  +
  a\Big\lfloor\frac{n}{a}\Big\rfloor
  +
  b\Big\lfloor\frac{n}{\delta}\Big\rfloor. 
\]
Then for all $n>m$,
\[
  w_n-w_m \notin S.
\]
\end{lemma}

\begin{proof}
Let 
\[
  A(n)=\Big\lfloor\frac{n}{a}\Big\rfloor\ \text{ and }\  B(n)=\Big\lfloor\frac{n}{\delta}\Big\rfloor.
\]
Then $w_n = n + aA(n) + bB(n)$. For $n>m$, we will study the {\em gaps} 
\begin{align}\label{eq:gap}
g(n,m)&:=w_n - w_m\\ \notag
  &= \ell + a\alpha + b\beta, \notag
\end{align}
where   $\ell := n-m\ge 1$,  $\alpha := A(n)-A(m)\ge 0$ and $\beta := B(n)-B(m)\ge 0$. We must prove that, for all $n>m$,  $g(n,m) \not\in\{a,b,a+b\}$. The proof is by contradiction. 

\smallskip\noindent
\emph{Case 1:} Suppose that $g(n,m)=a$.  Then $\ell + a\alpha + b\beta = a$
and hence $\ell + a(\alpha-1) + b\beta = 0.$  Since $\ell>0$ this forces $\alpha=\beta=0$, and hence $\ell=a$. But, by definition of the floor function, $\ell=n-m=a$ implies that $A(n)-A(m)=1$, contradicting $\alpha=0$. 
Thus $g(n,m)\neq a$. 

\smallskip\noindent
\emph{Case 2:} Suppose $g(n,m)=b=a+\delta$.
Then $\ell + a\alpha + b\beta = b$, so that
\begin{equation}\label{eq:b}
  \ell + a\alpha + b(\beta-1) = 0.
\end{equation}
If $\beta\ge 1$ then $b(\beta-1)\ge 0$, and since $a\alpha\ge 0$ and $\ell\ge 1$, we reach a contradiction. Hence $\beta=0$, which gives
 \begin{align}\label{eq:elladelta}
 \ell + a\alpha = a+\delta.
 \end{align}
 Since $\delta <2a$ and $\ell>0$, thus $\alpha\le 2$. 
\begin{itemize}
\item If $\alpha=0$, then $\ell=n-m=a+\delta>a$, so $A(n)-A(m)\ge 1$, contradicting $\alpha=0$. 
\item If $\alpha=1$, then $\ell=n-m=\delta$, which implies $B(n)-B(m)=1$, contradicting $\beta=0$.
\item If $\alpha=2$, then $\ell=n-m=\delta-a$. Since $a<\delta<2a$, we have $0<\delta-a<a$, so $n-m<a$, which forces $A(n)-A(m)\le 1$, contradicting $\alpha=2$.
\end{itemize}
Thus $g(n,m)\neq b$.

\smallskip\noindent
\emph{Case 3:} Suppose $g(n,m)=a+b=2a+\delta$. 
Then 
\begin{align}\label{eq:ellab}
\ell + a\alpha + b\beta = 2a+\delta. 
\end{align}
Since $\ell$ is positive we have
\begin{align}\label{eq:ab}
  a\alpha+(a+\delta)\beta < 2a +\delta .
\end{align}
Since $\alpha\ge 0$, we must have $\beta\le 1$. If $\beta=1$, then $\ell = a(1-\alpha)$, we must have $\alpha=0$, which implies $\ell=a$. But $\alpha=0$ together with $\ell=a$ is impossible as we saw in Case~1. Thus $\beta\ne 1$. The remaining case, $\beta=0$, by \eqref{eq:ab}, implies that $\alpha\le 3$. In this case, by \eqref{eq:ellab}, $n-m = a(2-\alpha) + \delta$.
\begin{itemize}
\item If $\alpha\le1$, then $ n-m \ge a+\delta> 2a$, and hence, by definition of the floor function, $A(n)-A(m)\ge 2$, contradicting $\alpha\le 1$.
\item If $\alpha=2$, then $n-m=\delta$, which implies $B(n)-B(m)=1$, contradicting $\beta=0$.
\item If $\alpha =3$, then $n-m=\delta-a<a$, which, by definition of floor function implies $1\ge A(n)-A(m)=3$, a contradiction. 
\end{itemize}
Thus $g(n,m)\neq a+b$.

We have demonstrated that,  for all $n>m$, $g(n,m)\not\in\{a,b,a+b\}$.
\end{proof}

Recall the affine translation of $\W_0$, the position set $\W_1 = \{w_n + a : n\ge 0\}$. By Lemma~\ref{lem:anticollision}, $\W_1$ satisfies the anti-collision property. Moreover, every position $x$ in this set has an option in $\mathcal W_0$, namely  $x-a$. If the set $\mathcal W_0$ constitutes the set of \P-positions, then, since $a=\min S$, Ferguson's pairing principle \cite{berlekamp2004winning} implies that every position in $\W_1$ has nim-value one, and there is no other position of value~one. Indeed, by induction, we can already deduce that each $x\in \W_1$ has nim-value one. (The other direction will follow only as we prove the main theorem.) 

\section{Higher nim-values}\label{sec:two}
In this section, we study positions with nim-values two and three. 
For all $n\ge 0$, let $v_{n} := w_{n+\delta} - b$. Then $\W_2 := \{v_{n}:n\ge 0\}$. 

The following lemma identifies the key arithmetic property of positions in $\W_2$. It makes explicit the effect of subtracting $b$ from any element of $\W_2$. Observe that, since $w_\delta=2b$ and by monotonicity, 
\begin{align}\label{eq:vnbge0}
v_n-b=w_{n+\delta}-2b\ge 0.
\end{align}
If $d=\delta -a$, then $0<d<a$. 
\begin{lemma}\label{lem:2b}
Consider the primitive quadratic regime. 
Let $n=qa+r$ with $0\le r<a$ and let $d=\delta-a$.
Then for all $n\ge \delta$,
\[
  w_n - 2b =
  \begin{cases}
    w_{n-\delta},     &\text{if } r \ge d,\\[6pt]
    w_{n-\delta} + a, &\text{if } r < d.
  \end{cases}
\]
\end{lemma}

\begin{proof}
Let
\[
  \Delta_a
  =
  \Big\lfloor\frac{n}{a}\Big\rfloor
  -
  \Big\lfloor\frac{n-\delta}{a}\Big\rfloor \]
  and
  \[
  \Delta_\delta
  =
  \Big\lfloor\frac{n}{\delta}\Big\rfloor
  -
  \Big\lfloor\frac{n-\delta}{\delta}\Big\rfloor.
\]
Then $\Delta_\delta =1$. Let $n=qa+r$ with $0\le r<a$. Then  $n-\delta = qa + (r-\delta)$. By $a<\delta<2a$, we get 
\[
  \Delta_a
  =
  \begin{cases}
    2, & r < \delta-a,\\[4pt]
    1, & r \ge \delta-a.
  \end{cases}
\]
Let us estimate 
\begin{align*}
  w_n - w_{n-\delta}
  &=
  \delta + a\Delta_a + b\Delta_\delta\\
  &=
  \delta + a\Delta_a + b.
\end{align*}

If $r\ge \delta-a$, then $\Delta_a=1$ and hence  $w_n -  w_{n-\delta}=2b$. If $r<\delta-a$, then $\Delta_a=2$ and hence $w_n - w_{n-\delta}= 2b + a$. Thus $w_n - 2b = w_{n-\delta}\in \W_0$, if $r\ge \delta-a$, and otherwise $w_n - 2b = w_{n-\delta}+a\in \W_1$.
\end{proof}

We also record the following consequence of Lemma~\ref{lem:anticollision}, which states a purely combinatorial fact. For any $v_n\in \W_2$, then, for all $n\ge 0$, by anti-collision in the set $\W_0$,
\begin{align}\label{eq:vna}
  v_n - a = w_{n+\delta} - b - a \not\in \W_0.
\end{align}

Recall the set $\W_3:=(\W_0-\delta)\setminus \W_0$. We will establish that every position in $\W_3$ has nim-value three; First we establish that all three options are available (as required for a position of nim-value three). 
\begin{lemma}\label{lem:boundary}
The following hold:
\begin{itemize}
\item[(i)] If $n < a$, then $w_n < \delta$.
\item[(ii)] If $a \leq n < \delta$, then $w_n - \delta \in \W_0$.
\item[(iii)] If $n \geq \delta$, then $w_n - \delta \geq a + b$.
\end{itemize}
In particular, every element $x\in \W_3$ satisfies $x \geq a+b$, so all three moves are 
available from any position in $\W_3$.
\end{lemma}
\begin{proof}
Item (i): if $n < a$ then $\lfloor n/a \rfloor = \lfloor n/\delta \rfloor = 0$, so 
$w_n = n < a < \delta$.

Item (ii): if $a \leq n < \delta$ then $\lfloor n/\delta \rfloor = 0$ and 
$\lfloor n/a \rfloor = 1$, so $w_n = n + a$ and $w_n - \delta = n - d$. Since 
$n < \delta = a + d$, we have $n - d < a$, and hence $\lfloor (n-d)/a \rfloor = 
\lfloor (n-d)/\delta \rfloor = 0$, giving $w_{n-d} = n - d$. Thus 
$w_n - \delta = w_{n-d} \in \W_0$.

Item (iii): if $n \geq \delta$ then $\lfloor \delta/a \rfloor = \lfloor \delta/\delta 
\rfloor = 1$, so $w_n\ge w_\delta = \delta + a + b$, by $(w_n)$ strictly increasing. Thus,
 $w_n - \delta \geq a + b$ for all $n \geq \delta$.

The final claim follows from (i) and (ii): no element of $W_3$ can arise from 
$n < a$ (as $w_n - \delta < 0$) or from $a \leq n < \delta$ (as $w_n - \delta \in W_0$), 
so every element of $W_3$ has $n \geq \delta$, and by (iii) satisfies $x \geq a+b$.
\end{proof}
\begin{lemma}\label{lem:two}
Every element of $\W_2$ has nim-value two.
\end{lemma}

\begin{proof}
We proceed by strong induction on position size. Assume that every position $x < v_n$ has nim-value $i$ if and only if $x \in \W_i$. For the ``anti-collision'' part, Lemma~\ref{lem:anticollision} applies to $\W_2$ as a translate of $\W_0$: no move connects two elements of $\W_2$.

Let us verify ``reachability''. By Lemma~\ref{lem:2b} and \eqref{eq:vnbge0}, $v_n - b = w_{n+\delta} - 2b \in \W_0 \cup \W_1$.  If $v_n - b \in \W_1$, then, by definition of $\W_1, v_n-b\ge a$, and thus, by Ferguson pairing,  $(v_n-b)-a \in \W_0$, with value zero, by induction. Otherwise, if $v_n - b \in \W_0$ then by induction, the value is zero. 

Thus, it suffices to find an option of $v_n$ to value one, given that 
\begin{align}\label{eq:vnbW0}
v_n - b \in \W_0. 
\end{align}
Let us study the option $v_n-a$.  By \eqref{eq:vna}, $v_n - a \notin \W_0$. Anti-collision of $\W_2$, by induction, implies $v_n - a \not\in \W_2$. So, by induction, $v_n-a $ has nim-value either one or three. If $v_n - a \in \W_1$  then it yields an option of value one by induction, and we are done. So, suppose position 
\begin{align}\label{eq:vnaW3}
v_n-a\in\W_3.
\end{align}
 We claim that this implies that there is an option $v_n-a-b\ge 0$ of value one (by induction). The bounding away from zero is clear by Lemma~\ref{lem:boundary}. Moreover, note that, by the assumption \eqref{eq:vnaW3}, induction and anti-collision, the position $(v_n-a)-b$ cannot have value three. And $v_n-(a+b)$ cannot have value two by induction and anti-collision of $\W_2$ (implied by Lemma~\ref{lem:anticollision}). Moreover it cannot have value zero, by anti-collision, since this is the case \eqref{eq:vnbW0}, where we assume that $v_n-b\in \W_0$. But then, since all other values are ruled out, the option $v_n-b-a$ must have value one.
\end{proof}


The second result of this section concerns the contribution of the set $\W_3$. We will prove that all those positions  have nim-value three. Recall that, by our convention, $\W_0-\delta\subset \mathbb N$. Note that $w_n-\delta \ge 0$ if and only if $n\ge a$. Let us begin with a technical reachability lemma, which concerns a special class of positions in $\W_3$. In Figure~\ref{fig:winning_model}, we have depicted the main idea of the proof.
 
 \begin{lemma}\label{lem:shift}
Let $n = qa + r\ge a$ with $0 \leq r < a$, and consider the position $w_n-\delta$. Given the provisos $r < d:=\delta-a$ and $w_n - \delta \notin \W_0$, the option $(w_n - \delta) -b   \in \W_0$.
\end{lemma}

\begin{proof}
We set out to prove that, given the provisos in the statement, then $w_n - w_{n-d} = 2b - a$, which implies the lemma. 

The proof is by way of contradiction, and to this purpose we define another representation of the index $n = p\delta +s$, with $0\le s<\delta$. We first demonstrate that, given the two provisos, then $s<d$. Suppose for contradiction that $s \geq d$. 
Since $n-d= p\delta+s-d$, then no multiple of $\delta$ lies in the interval $(n-d,n]$, namely $p\delta\le n-d\le n=p\delta+s<(p+1)\delta$. Hence  
\begin{align}\label{eq:zero}
\left\lfloor \frac{n}{\delta} \right\rfloor - \left\lfloor \frac{n-d}{\delta} \right\rfloor = 0.
\end{align}
Since $r < d < a$, the interval $(n-d, n]$ contains exactly one 
multiple of $a$, namely $qa$; this follows since $n-d=qa+r-d <qa \leq n = qa + r$,  
as $r < d$. Thus 
\begin{align}\label{eq:aone}
\left\lfloor \frac{n}{a} \right\rfloor - \left\lfloor \frac{n-d}{a} \right\rfloor = 1.
\end{align}
Altogether \eqref{eq:zero} and \eqref{eq:aone} give
\begin{align*}
w_n - w_{n-d} &= d + a \cdot 1 + b \cdot 0\\ &= d + a\\ &= \delta,
\end{align*}
and hence $ w_n - \delta = w_{n-d} \in \W_0$, contradicting the proviso $w_n - \delta \notin \W_0$. 
Thus $s < d$.

Since $s < d$, then $s+a<\delta$. Thus, 
\begin{align*}
\lfloor (n-d)/\delta \rfloor &= \lfloor (q\delta +s-\delta+a)/\delta \rfloor\\
&= \lfloor (q\delta +s+a)/\delta \rfloor-1\\ 
&= \lfloor n/\delta \rfloor -1, 
\end{align*}
and hence  $\lfloor n/\delta \rfloor - \lfloor (n-d)/\delta \rfloor = 1$. This will contribute the term 
$b$ to $w_n - w_{n-d}$, and by \eqref{eq:aone}, 
$\lfloor n/a \rfloor - \lfloor (n-d)/a \rfloor = 1$, which contributes the amount $a$. Altogether, we get 
\begin{align*}
w_n - w_{n-d} &= d + a \cdot 1 + b \cdot 1\\ 
&= \delta + b \\
&= 2b - a,
\end{align*}
and hence $w_{n-d} = w_n - 2b + a \in \W_0$.
\end{proof}

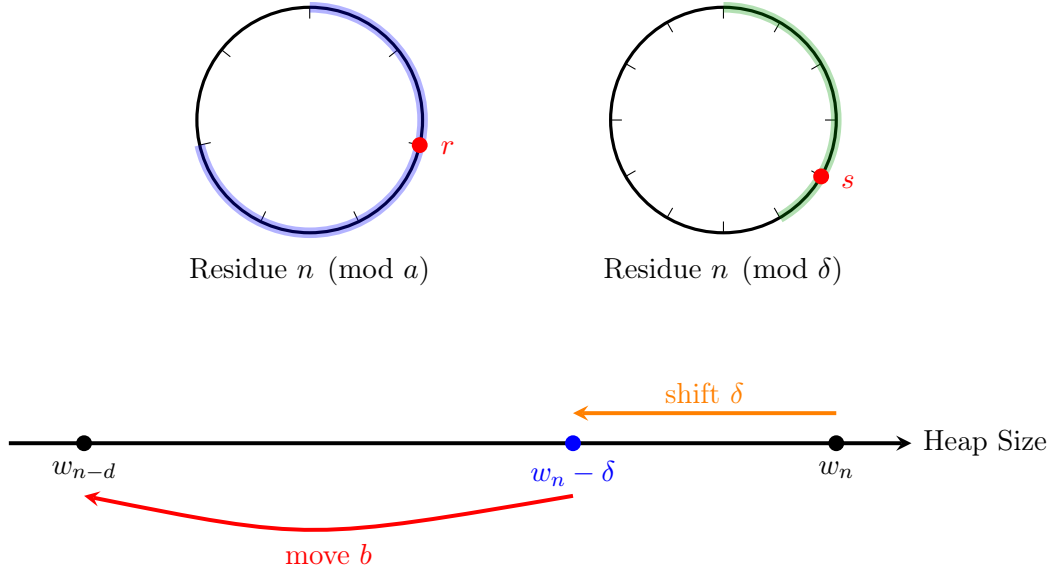
\begin{figure}
\begin{tikzpicture}[>=stealth]

    \begin{scope}[shift={(2.5,-1.2)}]
        \draw[line width=1.2pt] (0,0) circle (1.5cm);
        \foreach \k in {0,1,2,3,4,5,6}
            \draw (90-\k*360/7:1.35) -- (90-\k*360/7:1.5);

        \draw[line width=4pt, blue, opacity=0.3]
            (90:1.5) arc (90:-167.1:1.5);

        \fill[red] (90-2*360/7:1.5) circle (3pt);
        \node[red, right] at (90-2*360/7:1.65) {$r$};
        \node at (0,-2) {Residue $n \pmod{a}$};
    \end{scope}

    \begin{scope}[shift={(8,-1.2)}]
        \draw[line width=1.2pt] (0,0) circle (1.5cm);
        \foreach \k in {0,1,...,11}
            \draw (90-\k*30:1.35) -- (90-\k*30:1.5);

        \draw[line width=4pt, green!60!black, opacity=0.3]
            (90:1.5) arc (90:-60:1.5);

        \fill[red] (90-4*30:1.5) circle (3pt);
        \node[red, right] at (90-4*30:1.65) {$s$};
        \node at (0,-2) {Residue $n \pmod{\delta}$};
    \end{scope}

    \begin{scope}[shift={(-0.5,-5.5)}]
        \draw[->, line width=1.5pt] (-1,0) -- (11,0) node[right] {Heap Size};

        \fill (0,0)   circle (3pt) node[below=3pt] {$w_{n-d}$};
        \fill (10,0)  circle (3pt) node[below=3pt] {$w_n$};
        \fill[blue] (6.5,0) circle (3pt) node[below=3pt] {$w_n-\delta$};

        \draw[->, orange, line width=1.5pt] (10,0.4) -- (6.5,0.4);
        \node[orange, above] at (8.25,0.4) {shift $\delta$};

        \draw[->, red, line width=1.5pt]
            (6.5,-0.7) .. controls (3,-1.3) .. (0,-0.7);
        \node[red, below] at (3.25,-1.2) {move $b$};
    \end{scope}

\end{tikzpicture}
\caption{The Winning Model. The picture requires $w_n-\delta \not\in \W_0$. Zero is at 12 o'clock modulo $a$ or $\delta$ respectively; the wheels have $a=7$ and 
$\delta=12$ equally spaced ticks, respectively. The colored arcs span $d=5$ units on each wheel,
representing the zones of sufficently small residues. The red dots show residues $r=2$ and $s=4$, both satisfying $r,s<d$. Since both residues jump across the north pole when the index shifts by $d$, the bracket
expression arithmetic produces a combined heap change of $d+a+b=\delta+b$, which suffices to reach a \P-position by
subtracting $b$ from $w_n-\delta$.}\label{fig:winning_model}
\end{figure}

\begin{lemma}\label{lem:three}
Every element of $\W_3$ has nim-value three.
\end{lemma}
\begin{proof}
Recall that $\W_3:=(\W_0-\delta)\setminus \W_0$. Anti-collision holds since $\W_3$ is a subset of a linear shift of $\W_0$. Further, note that $\W_0-\delta=\W_0-b+a$, and hence, by subtracting $a$ from a position $x\ge a$ in $\W_3$, by induction, a player can move to nim-value two. Let us consider reachability to the smaller nim-values. Let $n\equiv r\pmod a$, and let $d=\delta-a$. By observing also Lemma~\ref{lem:boundary}, we record the following claims: 
\begin{itemize}
\item[(i)] If $w_n-\delta \in \W_3$ with $r<d$ then, subtracting $b$ gives value zero, while subtracting $(a+b)$ gives value one.
\item[(ii)] If $w_n-\delta \in \W_3$ with $r\ge d$ then, subtracting $b$ gives value one, while subtracting $(a+b)$ gives value zero.
\end{itemize}

In case (i), we use Lemma~\ref{lem:2b} to see that the value one option is correct, by induction. Indeed, $w_n-b+a-(b+a)=w_n-2b=w_{n-\delta}+a\in \W_1$. Therefore, it suffices to prove that, given the provisos $r<d$ and $w_n-\delta=w_n-b+a\not \in \W_0$, the position \begin{align}\label{eq:wn2ba}
w_n-2b+a=w_{n-d},
\end{align}
which by induction implies value zero. By Lemma~\ref{lem:shift}, equation \eqref{eq:wn2ba} holds, and we are done with this case.

In case (ii),  we reach value zero by subtracting $a+b$ from $w_n-\delta$. Indeed, by invoking the proviso $r\ge d$, Lemma~\ref{lem:2b} gives that $w_n-2b=w_{n-\delta}\in \W_0$ has value zero, by induction. Moreover, by instead subtracting $b$, Ferguson pairing gives value one at heap size $w_n-2b+a=w_{n-\delta}+a\in \W_1$, by induction, and since $a=\min S$. 
\end{proof}

We now understand the nim-values of the relevant sets. It remains to show that, in fact, $\W_0\sqcup\W_1\sqcup\W_2\sqcup\W_3=\N$, i.e.\ these four sets exhaust all positions. To prove this, it suffices to show that the number of positions of $\W_3$, in one heap-period, matches 
\begin{align}\label{eq:a2}
    a^2 = (3\delta +a)a - 3a\delta. 
\end{align} 
These are the remaining positions in one heap-period, given that the three sets $\W_0, \W_1$ and $\W_2$ have been accounted for. Let us explain. Obviously, the number of positions in a period of $\W_0$, or any translate of it, is $a\delta$ (this is the index-period). From this, a direct calculation from the closed form $w_n$ shows that the period length of the \P-positions (the heap-period) is:
\begin{align}\label{eq:adelta}
  w_{n+a\delta} - w_n \notag
  &=
  a\delta
  +
  a\Bigl(\Bigl\lfloor\frac{n+a\delta}{a}\Bigr\rfloor - \Bigl\lfloor\frac{n}{a}\Bigr\rfloor\Bigr)
  +
  (a+\delta)\Bigl(\Bigl\lfloor\frac{n+a\delta}{\delta}\Bigr\rfloor - \Bigl\lfloor\frac{n}{\delta}\Bigr\rfloor\Bigr)\\
  &=
  (3\delta + a)a. 
\end{align}
Note that, since $a\delta$ is a multiple of both $a$ and $\delta$, the two floor differences are $\delta$ and $a$, respectively. The following section isolates the core number-theoretic mechanism in establishing the count in \eqref{eq:a2}.

\section{The arithmetic structure of the $\delta$-collisions}\label{sec:numtheo}

In this section we isolate the arithmetic mechanism governing the $\delta$-collisions in the primitive quadratic regime. That is, we study the set $(\W_0-\delta) \cap \W_0$. We do this, because it will determine the size of $(\W_0-\delta)\setminus\W_0$. The reader is invited to study the three examples we provide in Section~\ref{sec:numexa}, in parallel with the technical details in this section. 

Recall that in this regime we assume $a<\delta<2a,   \gcd(a,\delta)=1,   d:=\delta-a$, 
and, for $n\ge 0$, we let 
\[
  w_n = n+ a\Bigl\lfloor\frac{n}{a}\Bigr\rfloor+(a+\delta)\Bigl\lfloor\frac{n}{\delta}\Bigr\rfloor.
\]
We write $\W_0=\{w_n:n\ge 0\}$ as before. For a subset $X\subset\mathbb N_0$, we write $|X|_h$ for the number
of elements of $X$ whose heaps lie in a single heap-period; in the primitive quadratic case this is the same as counting indices in a single index-period of length $a\delta$.

We are interested in the intersections of $\W_0$ with its $\delta$-shift
$\W_0-\delta=\{w_n-\delta:n\ge a\}$. Equivalently, we count all pairs $(m,n)$ with $n>m$ such that $w_n-w_m=\delta$.

Let the collision set be 
\begin{align}\label{eq:col}
C := \{\, w_n : \exists m<n \text{ with } w_n-w_m=\delta \,\}
= (\W_0-\delta)\cap \W_0 .
\end{align}
Let us state the main result of this section. 
\begin{thm}[$\delta$-Collision Counting]\label{thm:deltacoll}
Assume $a<\delta<2a$ and $\gcd(a,\delta)=1$. 
Then, in one heap-period of length $h=(3\delta+a)a$, the number of
collisions is  $|C|_h = ad$.
\end{thm}

We derive this result from a more detailed description of how many $\delta$-collisions each $(a+1)$-gap can take part in. To make the governing ideas more explicit, we divide the proof into four subsections. 

\subsection{Gaps and residue classes}
The arithmetic structure of the $\delta$-collisions relies on a regularity property of the {\em gap sequence} of the candidate $\mathcal P$-positions. 
Recall, for $n\ge 0$, 
\[
w_n = n + a\Big\lfloor\frac{n}{a}\Big\rfloor + b\Big\lfloor\frac{n}{\delta}\Big\rfloor.
\]
Define the gap sequence, for $n\ge 0$,
\[
g_n := w_{n+1}-w_n .
\]

The key observation is that the gap sequence has a very simple arithmetic
structure. Since the definition of $w_n$ involves only the two floor
functions $\lfloor n/a \rfloor$ and $\lfloor n/\delta \rfloor$, the gap
$g_n$ can change only when $n+1$ crosses a multiple of $a$
or a multiple of $\delta$. Consequently the gap sequence takes only four
possible values and is completely determined by the residue classes of
$n+1$ modulo $a$ and $\delta$. We record this observation explicitly.

\begin{lemma}[Gap Regularity]\label{lem:gapreg}

The gap sequence satisfies 
\[
g_n = 1 + aI_a(n) + bI_\delta(n),
\]
where
\[
I_a(n)=
\begin{cases}
1,& n+1\equiv 0 \pmod a,\\
0,& \text{otherwise},
\end{cases}
\qquad
I_\delta(n)=
\begin{cases}
1,& n+1\equiv 0 \pmod\delta,\\
0,& \text{otherwise}.
\end{cases}
\]

In particular the gap sequence is periodic with index-period $a\delta$, and the
occurrences of the special gaps $a+1$ and $b+1$ are uniquely determined
by the residue classes of $n+1$ modulo $a$ and $\delta$, respectively.
\end{lemma}

\begin{proof}
We compute
\[
g_n = w_{n+1}-w_n
= 1 + a\!\left(\Big\lfloor\frac{n+1}{a}\Big\rfloor-\Big\lfloor\frac{n}{a}\Big\rfloor\right)
  + b\!\left(\Big\lfloor\frac{n+1}{\delta}\Big\rfloor-\Big\lfloor\frac{n}{\delta}\Big\rfloor\right).
\]
The difference
\[
\Big\lfloor\frac{n+1}{a}\Big\rfloor-\Big\lfloor\frac{n}{a}\Big\rfloor
\]
equals $1$ exactly when $n+1$ is divisible by $a$, and $0$ otherwise; similarly for
the modulus $\delta$.  This yields the stated formula.
\end{proof}

By Lemma~\ref{lem:gapreg}, the set of the possible gap sizes is:
\begin{align}\label{eq:gapsizes}
  \{1,  a+1,  b+1, 
  b+a+1\}.
\end{align}

In particular,
\[
  g_n = a+1
  \quad\Longleftrightarrow\quad
  n+1 \equiv 0 \pmod a \
  \text{ and }\ 
  n+1 \not\equiv 0 \pmod\delta.
\]
Let $n+1=qa$ and reduce modulo $\delta$. Set $\rho \equiv qa \pmod{\delta}$. Since $\gcd(a,\delta)=1$, multiplication by $a$ is invertible modulo $\delta$. 

\begin{lemma}[Residue Permutation]\label{lem:permutation}
As $q$ runs through $1,\dots,\delta-1$, the residues $\rho \equiv qa \pmod{\delta}$ run bijectively through $\{1,\dots,\delta-1\}$. In particular, in one index-period of length $a\delta$ there are exactly $\delta-1$ gaps equal to $a+1$.
\end{lemma}
\begin{proof}
Since $\gcd(a,\delta)=1$, multiplication by $a$ is invertible modulo $\delta$. Hence the map $q \mapsto qa \pmod{\delta}$ permutes the nonzero residue classes. 
\end{proof}

Thus we can label each $(a+1)$-gap in a period by a unique residue $\rho\in\{1,\dots,\delta-1\}$.

\subsection{Collision windows and the parameters $L$ and $R$}

A \emph{collision window} is a stretch of consecutive gaps whose sum
equals $\delta=a+d$. Since $a<\delta<2a$, by \eqref{eq:gapsizes} any such collision window must contain exactly one
$(a+1)$-gap and exactly $d-1$ gaps of size one. Gaps greater than $\delta$ (i.e. $b+1$ and $a+b+1$) cannot appear inside a collision window; in fact, they terminate any consecutive run of $1$-gaps. Fix one particular $(a+1)$-gap. Let $L$ and $R$ denote the numbers of consecutive $1$-gaps immediately to its left and right, respectively; for an example, see Figure~\ref{fig:LR}.

\begin{lemma}\label{lem:LR}
Let the $(a+1)$-gap correspond to residue $\rho\in\{1,\dots,\delta-1\}$
as above. Then
\[
  L=\min(a-1,\rho-1),
  \qquad
  R=\min(a-1,\delta-\rho-1).
\]
\end{lemma}

\begin{proof}
Between two successive multiples of $a$ there are exactly $a-1$ indices, and hence at most $a-1$ gaps of size $1$ appear before the next large gap. Thus, $L$ and $R$ will have size $a-1$ unless a larger gap appears before the next multiple of $a$. 

Measuring from the nearest multiple of $\delta$ to the left, the residue $\rho$ counts how many steps one is away from
that multiple. Thus there can be at most $\rho-1$ consecutive indices before reaching that multiple from the right, and at most $\delta-\rho-1$ indices before reaching the next multiple of $\delta$ to the right. Combining these two types of obstructions yields the stated minima.
\end{proof}

\begin{figure}[h]
\centering
\begin{tikzpicture}[scale=0.8]

  \coordinate (w0) at (0,0);
  \coordinate (w1) at (1,0);
  \coordinate (w2) at (2,0);
  \coordinate (w3) at (3,0);
  \coordinate (w4) at (8,0);
  \coordinate (w5) at (9,0);
  \coordinate (w6) at (17,0);

  \draw[-] (w0) -- (w5);

  \foreach \p/\lab in {
    w0/{w_{n-3}}, w1/{w_{n-2}}, w2/{w_{n-1}},
    w3/{w_{n}},   w4/{w_{n+1}}, w5/{w_{n+2}}, w6/{w_{n+3}}
  }{
    \fill (\p) circle (2pt);
    \node[below=4pt,scale=0.7] at (\p) {$\lab$};
  }

  \draw (w0) -- (w1) node[midway,above=2pt,scale=0.8] {$1$};
  \draw (w1) -- (w2) node[midway,above=2pt,scale=0.8] {$1$};
  \draw (w2) -- (w3) node[midway,above=2pt,scale=0.8] {$1$};

  \draw[very thick] (w3) -- (w4)
    node[midway,above=2pt] {$a+1=5$};

  \draw (w4) -- (w5) node[midway,above=2pt,scale=0.8] {$1$};

  \draw[dashed] (w5) -- (w6);
  \node[above=2pt,scale=0.8] at ($(w5)!0.5!(w6)$) {$> \delta$};

  \draw[decorate,decoration={brace,amplitude=4pt,mirror}]
    ($(w0)+(0,-0.65)$) --
    ($(w3)+(0,-0.65)$)
    node[midway,below=6pt,scale=0.8] {$L=3$};

  \draw[decorate,decoration={brace,amplitude=4pt,mirror}]
    ($(w4)+(0,-0.65)$) --
    ($(w5)+(0,-0.65)$)
    node[midway,below=6pt,scale=0.8] {$R=1$};

  \draw[decorate,decoration={brace,amplitude=4pt},blue,thick]
    ($(w1)+(0,1)$) --
    ($(w4)+(0,1)$)
    node[midway,above=6pt,scale=0.8,blue] {$1+1+5=\delta$};

  \draw[decorate,decoration={brace,amplitude=4pt},red,thick]
    ($(w2)+(0,2.0)$) --
    ($(w5)+(0,2.0)$)
    node[midway,above=6pt,scale=0.8,red] {$1+5+1=\delta$};

\end{tikzpicture}
\caption{Consider $a=4$ and $\delta=7$. 
The gaps along the number line have sizes $1,1,1,5,1,12$ where $12>\delta$. 
The central gap of size $a+1=5$ is flanked by $L=3$ consecutive one-gaps on the left
and $R=1$ on the right; the large dashed gap ``$>\delta$'' prevents extending collision
windows further right. The blue and red braces show the two possible collision windows of
length $\delta=7$.}\label{fig:LR}
\end{figure}
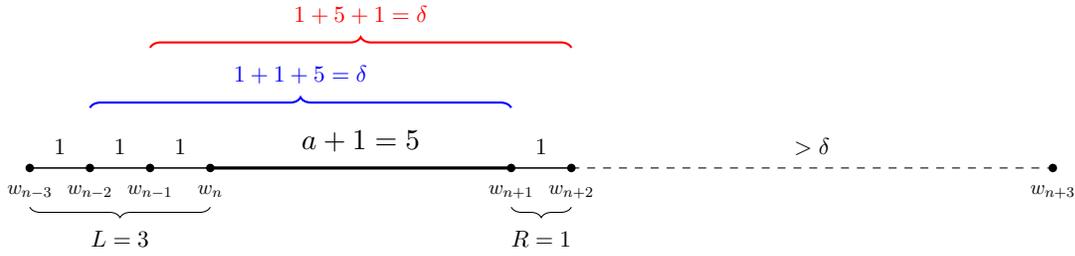

Let us denote by $c(\rho)$ the number of collision windows belonging to the $(a+1)$-gap at residue $\rho$. It must contain,  $0\le x\le d-1$ one-gaps from the left and $0\le d-1-x\le d-1$ one-gaps from the right, because $a+1+x+d-1-x=\delta$. 
By Lemma~\ref{lem:LR}, the admissible such $x$ values are those satisfying both 
\begin{equation}\label{eq:x-constraints}
  0\le x\le L\
  \text{ and }\ 
  0\le d-1-x\le R.
\end{equation} 

\subsection{Three residue regions and the distribution of $c(\rho)$}\label{sec:}
Recall that $\rho\equiv qa\pmod \delta$, where $q$ runs through $1,\ldots, \delta-1$. We now compute the number $c(\rho)$ of collision windows belonging to the $\rho$th $(a+1)$-gap explicitly. The numbers are computed by using the membership of the residue $\rho$ of the three relevant regions, two of which are restricted to the left and right, respectively. As before, let $d=\delta-a$ and recall that we are in the quadratic regime where $d<a$.

\medskip
\noindent
\textbf{Left-restricted:} $1\le \rho\le d$. Here $\rho-1\le d-1<a-1$, so
\[
  L=\min(a-1,\rho-1)=\rho-1,
  \qquad
  R=\min(a-1,\delta-\rho-1)= a-1.
\]
Thus the right side does not restrict. The constraints \eqref{eq:x-constraints} reduce to $0\le x\le \rho-1$. 
Hence $x$ can take the $\rho$ values $0,1,\dots,\rho-1$, and
\begin{align}\label{eq:left}
  c(\rho)=\rho.
\end{align}

\medskip
\noindent
\textbf{Unrestricted:} $d<\rho<a$. Now $\rho-1\ge d$ and $\delta-\rho-1\ge d-1$, so both $L$ and $R$ are at
least $d-1$, and the only restriction on $x$ comes from the requirement
$0\le d-1-x$. Hence $x$ can range freely from $0$ to $d-1$, and
\begin{align}\label{eq:middle}
  c(\rho)=d.
\end{align}

\medskip
\noindent
\textbf{Right-restricted:} $a\le \rho\le \delta-1$. 
We have 
\begin{align*}
  L&=\min(a-1,\rho-1)\\
  &=a-1.
\end{align*}
and
\begin{align*}
  R&=\min(a-1,\delta-\rho-1)\\
   &=\delta-\rho-1,
\end{align*}
since $\delta-\rho-1\le d-1<a-1$. Thus the constraints \eqref{eq:x-constraints} become
\[
  0\le d-1-x\le \delta-\rho-1.
\]
Which simplifies to
\[
  \rho-a\le x\le d-1, 
\]
which gives 
\begin{align}\label{eq:right}
c(\rho)=\delta-\rho.
\end{align}
By collecting the three regions, we obtain the following refinement.

\begin{lemma}[Primitive Collision Distribution]\label{lem:colldist}
Let $a<\delta<2a$ with $\gcd(a,\delta)=1$ and $d=\delta-a$. Then 
\[
  \sum_{\rho=1}^{\delta-1} c(\rho)
  =
  \sum_{k=1}^{d}2k + (a-d-1)d
  =
  ad.
\]
\end{lemma}

\begin{proof}
Since $\gcd(a,\delta)=1$, multiplication by $a$ permutes the nonzero residues modulo $\delta$. Hence as $q$ runs through $1,\ldots,\delta-1$, the residues $\rho \equiv qa \pmod{\delta}$ run through the set $\{1,2,\ldots,\delta-1\}$ exactly once. Thus it suffices to sum $c(\rho)$ for $\rho =1,\ldots,\delta-1$. From the three cases above we have
\[
c(\rho)=
\begin{cases}
\rho, & 1\le \rho\le d,\\
d, & d<\rho<a,\\
\delta-\rho, & a\le \rho\le\delta-1.
\end{cases}
\]

In the left region ($1\le \rho \le d$) the values are $1,2,\ldots,d$. In the middle region ($d<\rho<a$) there are $a-d-1$ residues and each contributes $d$. In the right region ($a\le \rho\le\delta-1$), the values $\delta-\rho$ run through 
$d,d-1,\ldots,1$. Thus the multiset of values is
\[
\{c(\rho):1\le \rho\le\delta-1\}
=
\{1,2,\ldots,d,
\underbrace{d,\ldots,d}_{a-d-1\ \text{times}},
d,\ldots,2,1\}.
\]

Summing these contributions gives
\[
\sum_{\rho = 1}^{\delta-1} c(\rho)
=
\sum_{k=1}^{d}2k + (a-d-1)d
=
ad.
\]
See Figure~\ref{fig:collision_distribution} for a geometric interpretation. 
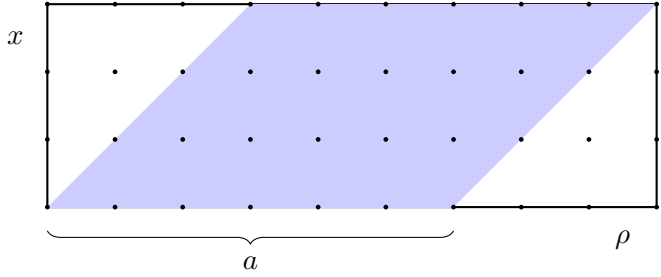
\begin{figure}[h]
\centering
\begin{tikzpicture}[scale=0.9]
\pgfmathtruncatemacro{\a}{7}
\pgfmathtruncatemacro{\d}{4}
\pgfmathtruncatemacro{\delta}{\a+\d}
\pgfmathtruncatemacro{\dminus}{\d-1}
\pgfmathtruncatemacro{\deltaminus}{\delta-1}

\draw[thick] (1,0) rectangle (\deltaminus,\dminus);

\begin{scope}
\clip (1,0) rectangle (\deltaminus,\dminus);
\fill[blue!20]
  (1,0) --
  (\deltaminus,\deltaminus-1) --
  (\deltaminus,\deltaminus-\a) --
  (1,1-\a) --
  cycle;
\end{scope}

\foreach \r in {1,...,\deltaminus}
  \foreach \x in {0,...,\dminus}
    \fill (\r,\x) circle (1pt);

\node[below] at ({(\delta)-1.5},-0.2) {$\rho$};
\node[left] at (0.8,{(\dminus)-.5}) {$x$};
\draw[decorate,decoration={brace,mirror, amplitude=5pt}]
  (1,-0.35) -- (\a,-0.35)
  node[midway,below=6pt] {$a$};

\end{tikzpicture}
\caption{The distribution of $c(\rho)$ across the three residue regions for $a < \delta < 2a$ and $d = \delta - a$. Each column $\rho \in \{1, \dots, \delta - 1\}$ contains exactly $c(\rho)$ admissible lattice points, where each point corresponds to a unique collision window containing the $(a+1)$-gap at residue $\rho$. The left triangle ($1 \le \rho \le d$) represents the left-restricted case where $c(\rho) = \rho$; the central rectangle ($d < \rho < a$) represents the unrestricted case where $c(\rho) = d$; and the right triangle ($a \le \rho \le \delta - 1$) represents the right-restricted case where $c(\rho) = \delta - \rho$. The total area (lattice point count) is $\sum_{\rho=1}^{\delta-1} c(\rho) = ad$.}\label{fig:collision_distribution}
\end{figure}
\end{proof}

\subsection{Proof of the $\delta$-collision counting theorem}

\begin{proof}[Proof of Theorem~\ref{thm:deltacoll}]
Each collision $w_n-w_m=\delta$ can be written uniquely as a sum of
consecutive gaps
\[
  \delta = g_m + g_{m+1} + \cdots + g_{n-1},
\]
so it corresponds to a unique collision window. Conversely, every collision window corresponds to a collision of this form.

In one index-period of length $a\delta$, the $(a+1)$-gaps are in bijection with residues $\rho\in\{1,\dots,\delta-1\}$, and for each residue $\rho$, $c(\rho)$ is the number of collision windows containing that gap. Thus the total number of collision windows (and therefore collisions) in one index-period equals
\[
  \sum_{\rho=1}^{\delta-1} c(\rho).
\]
By Lemma~\ref{lem:colldist} this sum equals $ad$,
which is the desired value.
\end{proof}

\section{Some numerical examples}\label{sec:numexa}
This section illustrates the arithmetic mechanism of Section~\ref{sec:numtheo} with three examples. 
The first shows the degenerate situation $d=1$, where every $(a+1)$-gap yields exactly one collision window. 
The second gives an explicit computation in the first non-trivial case $d=2$. The third demonstrates how the parameters $L$ and $R$ determine $c(\rho)$. 

\begin{example}
Let $a=3$ and $\delta=4$ (so $d=1$). Then one index-period has length $a\delta=12$. 
A direct computation from the closed form gives the gap sequence in one index-period as
\[
(g_n)=1,1,4,8,1,4,1,8,4,1,1,11.
\]
The gaps of size $a+1=4$ occur exactly $\delta-1=3$ times. Since $d=1$, a collision window
consists solely of a single $(a+1)$-gap and no additional $1$-gaps. Thus each such gap
contributes exactly one collision window, so $c(\rho)=1$ for all $\rho$ and
\[
\sum_\rho c(\rho)=3=ad.
\]
\end{example}
\begin{example}
Let $a=5$ and $\delta=7$ (so $d=2$). Then one index-period has length $a\delta=35$. 
From the closed form
\[
w_n
=
n
+
5\Big\lfloor\frac{n}{5}\Big\rfloor
+
12\Big\lfloor\frac{n}{7}\Big\rfloor,
\]
we obtain the gap sequence $g_n=w_{n+1}-w_n$, which in one index-period reads
\[
(g_n)=(1,1,1,1,6,1,13,1,1,6,1,1,1,13,6,1,1,1,1,6,
13,1,1,1,6,1,1,13,1,6,1,1,1,1,18).
\]

Here $a+1=6$, and we see exactly $\delta-1=6$ six-gaps, occurring at indices
\[
n\in\{4,9,14,19,24,29\}.
\]

Since $d=2$, each collision window must contain one $(a+1)$-gap and one
neighbouring $1$-gap. Inspecting the sequence shows that some six-gaps are
adjacent to two $1$-gaps while others are adjacent to only one. Writing $c(\rho)$
for the number of collision windows containing the $(a+1)$-gap at index $\rho$,
we obtain
\[
c(4)=2,\quad c(9)=2,\quad c(14)=1,\quad c(19)=1,\quad
c(24)=2,\quad c(29)=2.
\]

Hence among the $\delta-1=6$ many $(a+1)$-gaps, exactly two have $c(\rho)=1$ and
the remaining four have $c(\rho)=2$. Therefore
\[
\sum_\rho c(\rho)=1+1+2+2+2+2=10=ad.
\]
\end{example}

\begin{example}\label{ex:LR}
Let $a=4$ and $\delta=7$. Then $d=3$, so each collision window must contain
exactly one gap of size $a+1=5$ and exactly $d-1=2$ gaps of size $1$.
This example illustrates how the parameters $L$ and $R$ determine the number $c(\rho)$. 
The index-period has length $a\delta=28$. 
The gap formula reads
\[
g_n=1+4\,I_a(n)+11\,I_\delta(n),
\]
so the possible gap sizes are $1$, $5$, $12$, and $16$. Since $\delta=7$, only the
gaps $1$ and $5$ can appear inside a collision window; the gaps $12$ and $16$
exceed $\delta$ and therefore terminate any run of $1$-gaps.

The $5$-gaps occur precisely when $n+1$ is divisible by $4$ but not by $7$. 
In one index-period the multiples of $4$ are $4,8,12,16,20,24,28$, 
and among these only $28$ is divisible by $7$. Hence there are exactly 
$6=\delta-1$ gaps of size $5$, as predicted by Lemma~\ref{lem:colldist}. 

Consider the gap corresponding to $n+1=12$. Writing $12=qa$ with $q=3$ and
reducing modulo $\delta=7$ gives $12\equiv5\pmod7$, so $\rho=5$. We now compute
\[
  L=\min(a-1,\rho-1)=\min(3,4)=3,
\]
  and 
\[  
R=\min(a-1,\delta-\rho-1)=\min(3,1)=1.
\]
Thus three consecutive $1$-gaps lie to the left of this $5$-gap and one lies
to its right. Since $d-1=2$, a collision window must choose $x$ one-gaps on the left and
$2-x$ on the right. The constraints $$0\le x\le L=3,$$ and $$0\le 2-x\le R=1$$ 
force $x=1$ or $x=2$. Hence this $(a+1)$-gap contributes $c=2$ collision windows.
This agrees with the right-restricted case of the general formula (where $a\le \rho\le \delta-1$).
\end{example}

\section{The additive subtraction nimbers finally find their way home}\label{sec:gametheo}
The goal of Section~\ref{sec:numtheo} was to establish the following result. 
\begin{corollary}\label{cor:a2}
In one heap-period, $|\W_3|_h = a^2$. 
\end{corollary}
\begin{proof}
By Theorem~\ref{thm:deltacoll}, we have that $|C|_h = ad$, where $d=\delta -a$ and consequently 
\begin{align*}
  |\W_3|_h&=|W_0-\delta|_h - |C|_h\\
  &= a\delta - ad\\
  &= a^2, 
\end{align*}
since the \P-position sequence $(w_n)$ has index-period $a\delta$ and any affine translation has the same index-period.
\end{proof}

We conclude our study, with a restatement and proof of Theorem~\ref{thm:main}.

\begin{thmm}[Main Theorem]
Assume $a<\delta<2a$ and $\gcd(a,\delta)=1$. Then
\[
  \N = \W_0 \sqcup \W_1 \sqcup \W_2 \sqcup \W_3,
\]
and $\W_i$ is precisely the set of positions of nim-value $i$
for $i\in\{0,1,2,3\}$.
\end{thmm}

\begin{proof}
By Lemma~\ref{lem:anticollision}, each $\W_i$ satisfies anti-collision. We have verified reachability for all $\W_i$ with $i>0$, which, together with anti-collision, implies that the sets are pairwise disjoint. We apply induction. By Ferguson pairing, $\W_1=\W_0+a$ consists exactly of the nim-value one positions. By Lemma~\ref{lem:two}, $\W_2=\W_0-b$ consists only of nim-value two positions, and by Lemma~\ref{lem:three}, $\W_3 = (\W_0-\delta)\setminus \W_0$ consists only of  three-valued  positions. It remains to verify that all positions have been accounted for. 

We have observed $|\W_0|_h = |\W_1|_h = |\W_2|_h = a\delta$, and, via Corollary~\ref{cor:a2}, $|\W_3|_h = a^2$. In total, the disjoint union
\[
  \W_0 \sqcup \W_1 \sqcup \W_2 \sqcup \W_3
\]
contains, per period,
\[
  |\W_0|_h + |\W_1|_h + |\W_2|_h + |\W_3|_h
  =
  3a\delta + a^2
\]
positions, exactly matching the heap-period length in the primitive quadratic regime. Recall, the index-period is $a\delta$, and since $a\delta$ is simultaneously a multiple of $a$ and $\delta$, the direct computation in \eqref{eq:adelta} gives the corresponding period length $w_{n+a\delta} - w_n=(3\delta + a)a$.

Since the sum of the number of positions within one period length coincides with the heap-period length, we have determined the nim-values of all game positions. 
\end{proof}

\section{Open problems and discussion}\label{sec:future}
An obvious problem is to extend this study to generic additive subtraction games. The dual setting in the preprint \cite{Bhagat} indicates that the main complexity is already covered in this study. Let us mention some further problems.

\begin{question}
Can we have bracket expressions with more than three terms to describe \P-positions of a subtraction game (of larger polynomial degree)?
\end{question}
\begin{question}
    For what {\em nim-value seeds} in the sense of \cite{MiklosPost} can we find nim-value sequences on three bracket terms for additive games similar to this study? 
\end{question}

A related number theoretic problem is to study other affine translations of bracket expressions on a fixed number of terms and estimate the sizes of the corresponding collision sets. Let us elaborate this problem. 

\begin{question}[Shift-Collision Structure of Bracket Expression]
Let $a, \delta$ be positive integers with $\gcd(a,\delta) = 1$, and define the three term bracket expression
\[
    w_n \;:=\; n \;+\; a\!\left\lfloor \frac{n}{a} \right\rfloor
              \;+\; (a+\delta)\!\left\lfloor \frac{n}{\delta} \right\rfloor,
    \qquad n \in \mathbb{N}_0,
\]
with associated image set $\W = \{w_n : n \geq 0\} \subseteq \mathbb{N}_0$.
For a positive integer $k$, define the \emph{$k$-collision set}
\[
    \mathcal{C}(k) \;:=\; (\W-k) \cap \W 
    \;=\; \bigl\{\, w_n \in \W : w_n + k \in \W \bigr\},
\]
and let $|\mathcal{C}(k)|_h$ denote the number of elements of $\mathcal{C}(k)$
in one heap-period of length $(3\delta + a)a$.

\begin{enumerate}
    \item[(i)] For which values of $k$ is $\mathcal{C}(k)$ non-empty, and how does
    $|\mathcal{C}(k)|_h$ depend on $k$, $a$, and $\delta$?

    \item[(ii)] Does the three-region residue structure
    (left-restricted, unrestricted, right-restricted)
    established for $k = \delta$ in the primitive quadratic case
    $a < \delta < 2a$ generalise to arbitrary shifts $k$,
    or to the regimes $\delta \leq a$ and $\delta \geq 2a$?

    \item[(iii)] More generally, suppose $w_n$ is an arbitrary bracket expression.
 Characterise $\mathcal{C}(k)$ as a function of $k$, and study the collision density
    \[
        \rho(k) \;:=\; \lim_{N \to \infty}
        \frac{|\mathcal{C}(k) \cap [0,N]|}{N}.
    \]
    Is it always rational? Does $C$ admit a closed form analogous to
    $|\mathcal{C}(\delta)|_h = ad$ in the primitive quadratic setting?
\end{enumerate}
\end{question}

The sequence $w_n$ is not a Beatty sequence in the classical sense; 
originating with Rayleigh \cite{Rayleigh1894} and formalized by Beatty and
Uspensky \cite{Uspensky1927}, early research concerned mostly floor functions 
$\lfloor n\alpha \rfloor$ with irrational moduli $\alpha$, and partitioning problems of more than one such sequence. 
For the rational case, Fraenkel \cite{Fraenkel1969} characterised exactly when two sequences of the
form $\lfloor n\alpha + \beta \rfloor$ with rational $\alpha$ are disjoint and
partition $\mathbb{N}$; Graham \cite{Graham1963} gave a simple proof of
Uspensky's theorem that no partition of $\mathbb{N}$ into three or more
\emph{homogeneous} Beatty sequences on irrational moduli exists. Fraenkel's conjecture
\cite{GrahamOBryant2005, Tijdeman2000} asserts that the only partition of
$\mathbb{N}$ into $m \geq 3$ Beatty sequences with distinct moduli greater
than one is given by rational moduli of the form $\bigl\{(2^m-1)/2^k : 0 \leq k < m\bigr\}$;
this has been verified for $m \leq 7$ but remains open in general.
A game-theoretic study of partitions by rational Beatty sequences, connecting
Fraenkel's conjecture directly to combinatorial game theory, appears in
\cite{FraenkelLarsson2019}.

The sequence $w_n$ studied here is instead a bracket expression involving two rational
Beatty-type floor functions. Its image $\W_0$ does not arise from a
complementary Fraenkel pair, and the four sets $\W_0, \W_1, \W_2, \W_3$ that
partition $\mathbb{N}_0$ have densities $\delta/(3\delta+a)$,
$\delta/(3\delta+a)$, $\delta/(3\delta+a)$, $a/(3\delta+a)$, which do
not conform to the geometric Fraenkel family $2^{m-1-k}/(2^m - 1)$.
The collision problem, estimating $|\mathcal{C}(k)| = |(\W-k) \cap (\W)|$, therefore lies outside the scope of all of the above, and the residue-region and collision-window analysis of Section~\ref{sec:numtheo} appears to be new in this setting of bracket expressions on fixed number of terms.

\bibliographystyle{plain}
\bibliography{referencesAdditive}
\end{document}